\RequirePackage{fix-cm}
\documentclass[smallextended, envcountsame]{svjour3}
\smartqed
\journalname{}

\usepackage[unicode = false,
    pdftoolbar = true,
    colorlinks = true,
    linkcolor = blue,
    citecolor = blue,
    filecolor = black,
    urlcolor = blue,
    breaklinks = true]{hyperref}

\usepackage[utf8]{inputenc}
\usepackage{listings}  
\usepackage{enumitem}  
\usepackage{amsmath}   
\usepackage{amssymb}   
\usepackage{xcolor}    
\usepackage{graphicx}  
\usepackage[ruled]{algorithm2e}
\usepackage{verbatim}
\usepackage{geometry}
\usepackage[justification=centering]{caption}
\usepackage{mathtools}
\usepackage{float}
\usepackage{bm}
\usepackage{subcaption}
\usepackage{placeins}
\usepackage{titlesec}
\usepackage{multirow}
\usepackage{moresize}

\definecolor{gab}{HTML}{50c878}
\definecolor{pat}{HTML}{bf68ff}
\definecolor{sham}{HTML}{7faaff}

\allowdisplaybreaks[4]

\DeclareMathAlphabet{\mymathbb}{U}{BOONDOX-ds}{m}{n}

\makeatletter
\newcommand*{\rom}[1]{\expandafter\@slowromancap\romannumeral #1@}
\makeatother

\SetKwIF{If}{ElseIf}{Else}{if}{}{else if}{else}{end if}%
\allowdisplaybreaks[4]
\SetArgSty{textnormal}

\titleformat{\paragraph}
{\normalfont\normalsize\itshape}{\theparagraph}{1em}{}
\titlespacing*{\paragraph}
{0pt}{3.25ex plus 1ex minus .2ex}{1.5ex plus .2ex}

\usepackage[utf8]{inputenc}
\usepackage{listings}  
\usepackage{enumitem}  
\usepackage{amsmath}   
\usepackage{amssymb}   
\usepackage{xcolor}    
\usepackage{graphicx}  
\usepackage[ruled]{algorithm2e}
\usepackage{verbatim}
\usepackage{geometry}
\usepackage[justification=centering]{caption}
\usepackage{mathtools}
\usepackage{float}
\usepackage{bm}
\usepackage{subcaption}
\usepackage{placeins}
\usepackage{titlesec}
\usepackage{multirow}
\usepackage{moresize}
\usepackage{etoolbox}
\usepackage{cases}

\allowdisplaybreaks[4]

\DeclareMathAlphabet{\mymathbb}{U}{BOONDOX-ds}{m}{n}

\DeclarePairedDelimiter\floor{\lfloor}{\rfloor}

\SetKwIF{If}{ElseIf}{Else}{if}{}{else if}{else}{end if}%
\allowdisplaybreaks[4]
\SetArgSty{textnormal}

\makeatletter
\renewenvironment{quotation}
   {\list{}{\small
            \listparindent 1.5em%
            \rightmargin   \leftmargin
            \parsep        \z@ \@plus\p@}%
    \item\relax}
   {\endlist}
\makeatother

\title{Determining inscribability of polytopes via rank minimization based on slack matrices}

\author{\hspace{-0.3cm}
\begin{tabular}{cccc}
    Yiwen Chen, & Jo\~ao Gouveia, & Warren Hare, & Amy Wiebe
\end{tabular}}
\institute{Yiwen Chen \at 
    Department of Mathematics, University of British Columbia, Kelowna, British Columbia, V1V 1V7, Canada. Chen's research is partially funded by the Natural Sciences and Engineering Research Council of Canada (les recherches de Chen sont partiellement financ\'ees par le Conseil de recherches en sciences naturelles et en g\'enie du Canada), Discover Grant \#2023-03555. (\email{yiwchen@student.ubc.ca})
\and Jo\~ao Gouveia \at
    CMUC, Department of Mathematics, University of Coimbra, 3001-454 Coimbra, Portugal.  Gouveia's research is partially supported by Centro de Matemática da Universidade de Coimbra (CMUC), funded by the Portuguese Government through FCT/MCTES, DOI 10.54499/UIDB/00324/2020. (\email{jgouveia@mat.uc.pt})
\and Warren Hare \at
    Department of Mathematics, University of British Columbia, Kelowna, British Columbia, V1V 1V7, Canada.  Hare's research is partially funded by the Natural Sciences and Engineering Research Council of Canada (les recherches de Hare sont partiellement financ\'ees par le Conseil de recherches en sciences naturelles et en g\'enie du Canada), Discover Grant \#2023-03555. (\email{warren.hare@ubc.ca})
\and Amy Wiebe \at
    Department of Mathematics, University of British Columbia, Kelowna, British Columbia, V1V 1V7, Canada.  Wiebe's research is partially funded by the Natural Sciences and Engineering Research Council of Canada (les recherches de Wiebe sont partiellement financ\'ees par le Conseil de recherches en sciences naturelles et en g\'enie du Canada), Discovery  Grant \#2024-04643. (\email{amy.wiebe@ubc.ca})
}
\date{\today}

\usepackage{amsopn}

\begin{document}

\maketitle

\begin{abstract}
A polytope is inscribable if there is a realization where all vertices lie on the sphere.  In this paper, we provide a necessary and sufficient condition for a polytope to be inscribable.  Based on this condition, we characterize the problem of determining inscribability as a minimum rank optimization problem using slack matrices.  We propose an SDP approximation for the minimum rank optimization problem and prove that it is tight for certain classes of polytopes.  Given a polytope, we provide three algorithms to determine its inscribability.  All the optimization problems and algorithms we propose in this paper depend on the number of vertices and facets but are independent of the dimension of the polytope.  Numerical results demonstrate our SDP approximation's efficiency, accuracy, and robustness for determining inscribability of simplicial polytopes of dimensions $4\le d\le 8$ with vertices $n\le 10$, revealing its potential in high dimensions.
\end{abstract}

\section{Introduction}
    We consider whether a given polytope is {\em inscribable}, that is, whether a combinatorially equivalent polytope exists for the given polytope such that all vertices are on the sphere.  This problem dates back to 1832 when Steiner \cite{steiner1832systematische} asked whether all 3-dimensional polytopes are inscribable.  In 1928, Steinitz \cite{steinitz1928isoperimetrische} provided a negative answer to Steiner’s question.  A characterization of 3-dimensional inscribable polytopes was found by Rivin~\cite{rivin1996characterization} in 1996.
    
    Compared to the inscribability of 3-dimensional polytopes, much less is known about higher-dimensional polytopes.  In particular, an open question raised by Padrol and Ziegler \cite{padrol2016six} in 2016 reads as follows:
    \begin{quotation}
        Question 2.2: Find strong necessary and sufficient conditions for the inscribability of higher-dimensional polytopes.
    \end{quotation}

    Using nonlinear optimization techniques, Firsching \cite{firsching2017realizability} provided a complete enumeration of simplicial 4-polytopes with up to 10 vertices and determined the inscribability of almost all of them.  However, their method is likely to be inefficient for determining the inscribability of higher dimensional polytopes, since it requires solving a system of polynomial equations that increases in size, degree, and number of variables with dimension.
    
    In this paper, we provide a new characterization of the inscribability of a polytope.  Our characterization is based on {\em slack matrices} of polytopes.  Given vertex and facet representations of a polytope, the $(i,j)$-th element of the slack matrix is the value of vertex $i$ evaluated at facet inequality $j$ (see Section \ref{sec:characterize} for more details).  Slack matrices have demonstrated effective use in the theory of extended formulations of polytopes \cite{braun2015approximation,fiorini2012linear,rothvoss2017matching,yannakakis1988expressing} and in characterizing the realization space of polytopes \cite{gouveia2019slack,gouveia2023combining}.  Using slack matrices, we provide a necessary and sufficient condition for the inscribability of a polytope.  We show that our condition can be easily checked by solving a particular rank minimization problem.
    
    In order to solve the minimum rank optimization problem, we introduce and analyze a semi-definite programming (SDP) approximation for it.  In particular, we prove that the SDP approximation is tight for many classes of polytopes.  We design experiments to demonstrate the efficiency, accuracy, and robustness of the SDP approximation.  We then provide less expensive heuristics to tune the SDP parameters and compare their results.
    
    It is worth mentioning that the minimum rank optimization problem and the SDP approximation we propose are independent of the dimension of the polytope, which is an advantage over the method given by Firsching \cite{firsching2017realizability}.  We run our experiments on simplicial polytopes of dimensions $4\le d\le 8$ with vertices $n\le 10$.  For each test case, the SDP approximation with 10 iterations of tuning parameters finishes in 10 seconds.  This suggests our method has the potential to be efficient in high dimensions.
    
    The remainder of this paper is organized as follows.  We end this section with remarks on our notation.  In Section \ref{sec:characterize}, we first review the definition of slack matrices.  Then, we present the main theorem of this paper, which provides a necessary and sufficient condition for the inscribability of a polytope.  We also explain how this inscribability problem can be formulated as an equivalent minimum rank optimization problem.  In Section \ref{sec:sdp}, we propose an SDP approximation for the minimum rank optimization problem introduced in Section \ref{sec:characterize}.  We prove that the SDP approximation is tight for certain classes of polytopes.  In Section \ref{sec:algs}, we present three algorithms for approximating the minimum rank problem.  In Section \ref{sec:numexp}, we design numerical experiments to compare the performance of the three algorithms.  Section \ref{sec:conclusion} concludes our work covered by this paper and suggests some potential future directions.

\subsection{Notation}
    For a matrix $A\in\mathbb{R}^{m\times n}$, we use $A_{ij}$ to denote the $(i,j)$-th element of $A$.  We use $A=[a_1\cdots a_n]$ to denote the column representation of $A$.  For real symmetric matrices $A$, we use $A\succeq 0$ to mean that $A$ is positive semi-definite.  We use $I_n$ to denote the identity matrix of order $n$ and $\mymathbb{1}_{m\times n}$ to denote the all-one matrix in $\mathbb{R}^{m\times n}$.  We use $\mymathbb{0}_n$ and $\mymathbb{1}_n$ to denote the all-zero and all-one vectors in $\mathbb{R}^{n}$, respectively.  For $i\in\{1,...,n\}$, we use $e_i^{(n)}\in\mathbb{R}^n$ to denote the $i$-th coordinate vector in $\mathbb{R}^n$. For a set $T$, we use $|T|$ to denote the cardinality of $T$.

\section{Characterizing inscribability using slack matrices}\label{sec:characterize}
    In this section, we first review the definition of a slack matrix and some of its basic properties.  Then, we give the main theorem of this paper, which characterizes the inscribability of a polytope using slack matrices.  Based on the theorem, we show how the inscribability problem can be formulated as an equivalent minimum rank optimization problem.
    
    For a $d$-polytope $P$, suppose that the $\mathcal{V}$-representation of $P$ is $P=\mathrm{conv}\{v_1,\ldots,v_n\}$ and one $\mathcal{H}$-representation of $P$ is $P=\{x\in\mathbb{R}^d:H^\top x\le c\}$ where $H\in\mathbb{R}^{d\times m}$ and $c\in\mathbb{R}^m$.  Let $V=[v_1\cdots v_n]\in\mathbb{R}^{d\times n}$.  Then a slack matrix $S\in\mathbb{R}^{n\times m}$ of $P$ is defined by
    \begin{equation*}
        S = \begin{bmatrix}
            \mymathbb{1}_n & V^\top
        \end{bmatrix}\begin{bmatrix}
            c^\top\\
            -H
        \end{bmatrix}.
    \end{equation*}

    Denote $H=[h_1\cdots h_m]$ and $c=[c_1\cdots c_d]^\top$, where each column of $H$ is a facet normal.  Notice that all elements of $S$ are nonnegative since $S_{ij}=c_j-h_j^\top v_i\ge 0$ for all $i=1,\ldots,n$ and $j=1,\ldots,m$.  As each facet inequality is invariant under positive scaling on both sides, any positive scaling of the columns of $S$ gives a slack matrix of $P$.

    We note that the combinatorial type of a polytope as well as the support of its slack matrix are invariant under projective transformations.
    \begin{lemma}\label{lem:projpreservecombtype}
        Projective transformations preserve the combinatorial type of a polytope and the support of its slack matrix.
    \end{lemma}
    \begin{proof}
        See \cite{gouveia2019slack} and \cite[Corollary 1.5]{gouveia2017four}.
        
    $\hfill\qed$
    \end{proof}
    Based on Lemma \ref{lem:projpreservecombtype}, the next lemma implies that, if a polytope is inscribable, then there exists an inscription such that the origin is in the interior.
    \begin{lemma}(Inspired by \cite[p. 285]{grunbaum2003convex})\label{lem:0ininterior}
        If $P$ is inscribed in a sphere, then we can apply a projective transformation that carries the sphere onto itself and maps an interior point of $P$ to the origin.  
    \end{lemma}
    \begin{proof}
        Without loss of generality, we suppose that $P$ is inscribed in the unit sphere $\mathcal{S}^{d-1}$ and $\alpha e_1^{(d)}$ is in the interior of $P$, where $|\alpha|<1$.  Inspired by \cite[p. 285]{grunbaum2003convex}, we consider the projective transformation $g:\mathbb{R}^d\to\mathbb{R}^d$ defined by
        \begin{equation*}
            g(x) = g(x_1,\ldots,x_d) = \left[\frac{x_1-\alpha}{1-\alpha x_1},\frac{\sqrt{1-\alpha^2}x_2}{1-\alpha x_1},\ldots,\frac{\sqrt{1-\alpha^2}x_d}{1-\alpha x_1}\right]^\top.
        \end{equation*}
        Then, $g(x)\in\mathcal{S}^{d-1}$ for all $x\in\mathcal{S}^{d-1}$ and $g(\alpha e_1^{(d)})=\mymathbb{0}_d$.

    $\hfill\qed$
    \end{proof}
    Therefore, without loss of generality, we can assume that $\mymathbb{0}_d$ is in the interior of $P$ and $c=\mymathbb{1}_m$, so the slack matrix of $P$ is $S=\mymathbb{1}_{n\times m}-V^\top H$.  Note that with this choice, the columns of $H$ are the vertices of the polytope that is polar to the polytope whose vertices are the columns of $V$.

    The following theorem characterizes the inscribability of a polytope by using slack matrices.
    \begin{theorem}\label{thm:inscri}
        Given a $d$-polytope $P$, it is inscribable if and only if we can construct a matrix
        \begin{equation*}
            X = \begin{bmatrix}
                1 & \mymathbb{1}_n^\top & \mymathbb{1}_m^\top\\
                \mymathbb{1}_n & A & S\\
                \mymathbb{1}_m & S^\top & B
            \end{bmatrix} \succeq 0
        \end{equation*}
        of rank $d+1$, with all diagonal elements of $A$ being the same constant, all elements of $S$ being nonnegative, and $S$ having the same support as a slack matrix of $P$.
    \end{theorem}
    \begin{proof}
        We first assume that the $d$-polytope $P$ is inscribable.  Moreover, we suppose that $P$ inscribed in a sphere with radius $a$, with vertices $v_1,\ldots,v_n\in\mathbb{R}^d$ and facets cut out by inequalities $1-h_1^\top x\ge 0,\ldots,1-h_m^\top x\ge 0$ with $h_1,\ldots,h_m\in\mathbb{R}^d$.  Denote matrices $V=[v_1\cdots v_n]\in\mathbb{R}^{d\times n}$ and $H=[h_1\cdots h_m]\in\mathbb{R}^{d\times m}$.  Then, we have $v_i^\top v_i=a^2,i=1,\ldots,n$.  Since $P$ is a $d$-polytope, we have $\mathrm{rank}([\mymathbb{1}_n~V^\top])=\mathrm{rank}([\mymathbb{1}_m~-H^\top])=d+1$ and thus the slack matrix of $P$, denoted by $S$, has rank $d+1$ \cite[Theorem 14]{gouveia2013nonnegative}.  
        
        Denote the matrix
        \begin{equation*}
            W = \begin{bmatrix}
                1 & \mymathbb{0}_d^\top\\
                \mymathbb{1}_n & V^\top\\
                \mymathbb{1}_m & -H^\top
            \end{bmatrix}.
        \end{equation*}
        Notice that $W$ also has rank $d+1$.  Then,
        \begin{equation*}
            WW^\top = \begin{bmatrix}
                1 & \mymathbb{1}_n^\top & \mymathbb{1}_m^\top\\
                \mymathbb{1}_n & \mymathbb{1}_{n\times n}+V^\top V & S\\
                \mymathbb{1}_m & S^\top & \mymathbb{1}_{m\times m}+H^\top H
            \end{bmatrix}\succeq 0
        \end{equation*}
        has rank $d+1$ and all diagonal elements in the block $\mymathbb{1}_{n\times n}+V^\top V$ equal $a^2+1$, and thus is the desired matrix $X$.

        For the other direction, suppose that there exists a matrix
        \begin{equation*}
            X = \begin{bmatrix}
                1 & \mymathbb{1}_n^\top & \mymathbb{1}_m^\top\\
                \mymathbb{1}_n & A & S\\
                \mymathbb{1}_m & S^\top & B
            \end{bmatrix} \succeq 0
        \end{equation*}
        of rank $d+1$, with all diagonal elements of $A$ being the same constant, all elements of $S$ being nonnegative, and $S$ having the same support as a slack matrix of $P$.  Since $X\succeq 0$ and has rank $d+1$, there exist a matrix $M\in\mathbb{R}^{(n+m+1)\times (d+1)}$ with rank $d+1$ such that $X=MM^\top$.  Denote 
        \begin{equation*}
            M = \begin{bmatrix}
                r_1^\top\\
                \vdots\\
                r_{n+m+1}^\top
            \end{bmatrix}.
        \end{equation*}
        Then we have $r_1^\top r_1=1$.  Thus, there exists an orthogonal matrix $Q\in\mathbb{R}^{(d+1)\times (d+1)}$ such that $Qr_1=e_1^{(d+1)}$.  Since $MQ^\top QM^\top=MM^\top=X$, without loss of generality, we suppose that $r_1=e_1^{(d+1)}$.  Then from the first row of $X$ we get $r_1^\top r_i=1,i=1,\ldots,n+m+1$, which implies that the first element of all $r_i,i=1,\ldots,n+m+1$ must equal 1.  That is, matrix $M$ has the form
        \begin{equation*}
            M = \begin{bmatrix}
                1 & \mymathbb{0}_d^\top\\
                \mymathbb{1}_n & M_1^\top\\
                \mymathbb{1}_m & -M_2^\top
            \end{bmatrix},
        \end{equation*}
        for some $M_1\in\mathbb{R}^{d\times n}$ and $M_2\in\mathbb{R}^{d\times m}$.  
        
        Denote $M_1=[v_1\cdots v_n]$ and $M_2=[h_1\cdots h_m]$.  We claim that polytope $P_M=\mathrm{conv}\{v_1,\ldots,v_n\}$ is a $d$-polytope inscribed in a sphere, with vertices $v_1,\ldots,v_n$ and facets cut out by inequalities $1-h_1^\top x\ge 0,\ldots,1-h_m^\top x\ge 0$, and is combinatorially equivalent to $P$, which implies that $P$ is inscribable.  
        
        Indeed, by \cite[Lemma 3.1]{gouveia2019slack}, $S$ has rank no less than $d+1$.  Since $\mathrm{rank}(S)\le\mathrm{rank}(X)=d+1$, we get $\mathrm{rank}(S)=d+1$.  From the fact that $\mathrm{rank}(X)=\mathrm{rank}(S)=d+1$ and the structure of $X$, we have $\mymathbb{1}_n$ in the column span of $S$.  Therefore, from \cite[Theorem 2.2]{gouveia2019slack}, we obtain that $S$ is a slack matrix of some realization of $P$. 

        Using \cite[Lemma 5 and Theorem 6]{gouveia2013nonnegative}, we have $S$ is a slack matrix of the cone generated by $\begin{bmatrix}
            1\\
            v_1
        \end{bmatrix},\ldots,\begin{bmatrix}
            1\\
            v_n
        \end{bmatrix}$ with an $\mathcal{H}$-representation formed by $\begin{bmatrix}
            1\\
            h_1
        \end{bmatrix},\ldots,\begin{bmatrix}
            1\\
            h_m
        \end{bmatrix}$.  Thus, from \cite[Theorem~14]{gouveia2013nonnegative} and the paragraph before it, we have $\{1\}\times P_M$ is isomorphic to $\mathrm{conv}(\mathrm{rows}(S))$ and so isomorphic to $P$, which implies that $P_M$ is isomorphic to $P$.  The proof is complete by noticing that all diagonal elements of $A= \mymathbb{1}_{n\times n}+M_1^\top M_1$ are constant which implies that $P_M$ is inscribed in a sphere.

    $\hfill\qed$
    \end{proof}

    We note that the proof of Theorem \ref{thm:inscri} implies a way to extract vertices $v_1,\ldots,v_n$ from $X$.  We will use this technique in the numerical experiments presented in Section \ref{sec:numexp}.

    Based on Theorem \ref{thm:inscri}, we can formulate the inscribability problem of a polytope as an equivalent minimum rank optimization problem.  Let $I^z=\{(i,j):1-h_j^\top v_i=0,i=1,\ldots,n,j=1,\ldots,m\}$ be the index set of zeros in the slack matrix.  Consider the following optimization problem
    \begin{align}\label{pro:ori}
        \begin{split}
            \min_X~~~&\mathrm{rank}(X)\\
            s.t.~~~& X = 
            \begin{bmatrix}
                1 & \mymathbb{1}_n^\top & \mymathbb{1}_m^\top\\
                \mymathbb{1}_n & A & S\\
                \mymathbb{1}_m & S^\top & B
            \end{bmatrix} \succeq 0\\
            & S_{ij} = 0,~\text{if}~(i,j)\in I^z\\
            & S_{ij} > 0,~\text{if}~(i,j)\notin I^z\\
            & A_{ii} = 2, i = 1,\ldots,n,
        \end{split}
    \end{align}
    where constraints $A_{ii} = 2, i = 1,\ldots,n$ come from the fact that we can assume all inscribable polytopes are inscribed in the unit sphere, without loss of generality.  Notice that the matrix $S$ in problem \eqref{pro:ori} has the same support as a slack matrix of a $d$-polytope and so the rank of $S$ is no less than $d+1$ \cite[Lemma 3.1]{gouveia2019slack}.  Thus, the minimum of problem \eqref{pro:ori} cannot be less than $d+1$ and $P$ is inscribable if and only if the minimizer of problem \eqref{pro:ori} has rank $d+1$.

    To the best of our knowledge, there is not an efficient direct method to solve problem~\eqref{pro:ori}.  Therefore, in the following sections, we provide a few methods to approximately solve problem \eqref{pro:ori}. Section \ref{sec:sdp} proposes a semi-definite programming (SDP) approximation, and Section~\ref{sec:algs} presents some algorithms to solve problem \eqref{pro:ori} numerically.

\section{SDP approximation}\label{sec:sdp}
    In this section, we provide an SDP approximation to approximate the minimum rank optimization problem \eqref{pro:ori} and prove that for certain classes of polytopes, the solution to the SDP approximation gives a solution to problem \eqref{pro:ori}.

\subsection{Problem formulation}
    Recall that the trace function is the convex envelope of the rank function over the set of matrices with norm bounded by one \cite{fazel2002matrix}.  Therefore, problem \eqref{pro:ori} can be approximated by
    \begin{align*}
    \begin{split}
        \min_X~~~&\mathrm{tr}(X)\\
        s.t.~~~& X = 
        \begin{bmatrix}
            1 & \mymathbb{1}_n^\top & \mymathbb{1}_m^\top\\
            \mymathbb{1}_n & A & S\\
            \mymathbb{1}_m & S^\top & B
        \end{bmatrix} \succeq 0\\
        & S_{ij} = 0,~\text{if}~(i,j)\in I^z\\
        & S_{ij} > 0,~\text{if}~(i,j)\notin I^z\\
        & A_{ii} = 2, i = 1,\ldots,n,
    \end{split}
    \end{align*}
    The constraints $S_{ij} > 0,(i,j)\notin I^z$ do not create a closed constraint set, making them difficult to work with.  These constraints can be replaced with penalties $\lambda_{ij}S_{ij},(i,j)\notin I^z$ with weights $\lambda_{ij}\ge 0$, resulting in the following SDP problem
    \begin{align}\label{pro:sdp_ori}
    \begin{split}
        \min_X~~~&\mathrm{tr}(X) - \sum_{(i,j)\notin I^z}\lambda_{ij}S_{ij}\\
        s.t.~~~& X = 
        \begin{bmatrix}
            1 & \mymathbb{1}_n^\top & \mymathbb{1}_m^\top\\
            \mymathbb{1}_n & A & S\\
            \mymathbb{1}_m & S^\top & B
        \end{bmatrix} \succeq 0\\
        & S_{ij} = 0,~\text{if}~(i,j)\in I^z\\
        & A_{ii} = 2, i = 1,\ldots,n.
    \end{split}
    \end{align}
    Note that problem \eqref{pro:sdp_ori} is an approximation of problem \eqref{pro:ori}. 
    
    We want to explore when the solution to problem \eqref{pro:sdp_ori} gives a solution to problem \eqref{pro:ori}.  Our strategy is as follows: First, we write problem \eqref{pro:sdp_ori} in terms of Schur complements and give the dual problem.  Then, we provide primal and dual feasible solutions for certain classes of polytopes and verify strong duality.
    
    We start by introducing the following classic result from linear algebra, the proof of which can be found in, e.g., \cite{boyd2004convex,lemon2016low}.
    \begin{lemma}\label{lem:psdschur}
        Suppose $M$ is a symmetric matrix of the form
        \begin{equation*}
            M=
            \begin{bmatrix}
                M_1 & M_2\\
                M_2^\top & M_3
            \end{bmatrix}.
        \end{equation*}
        Then $M\succeq 0$ if and only if $M_3\succeq 0$, $\mathrm{range}(M_2^\top)\subseteq\mathrm{range}(M_3)$, and $M_1-M_2 M_3^\dagger M_2^\top\succeq 0$.
    \end{lemma}
    
    Denote $\mathrm{tr}(M_1^\top M_2)=\langle M_1, M_2\rangle$.  According to Lemma \ref{lem:psdschur}, problem \eqref{pro:sdp_ori} is equivalent to the following problem
    \begin{align}\label{pro:sdp_primal}\tag{P}
    \begin{split}
        \min_{A,B,S}~~~&\langle A, I_n\rangle + \langle B, I_m\rangle - \sum_{(i,j)\notin I^z}\lambda_{ij}S_{ij} + 1\\
        s.t.~~~& 
        \begin{bmatrix}
            A & S\\
            S^\top & B
        \end{bmatrix} - \mymathbb{1}_{(n+m)\times (n+m)}\succeq 0\\
        & S_{ij} = 0,~\text{if}~(i,j)\in I^z\\
        & A_{ii} = 2, i = 1,\ldots,n.
    \end{split}
    \end{align}
    
    Let $u=[u_1\cdots u_n]^\top\in\mathbb{R}^n$ and $w=[w_1\cdots w_{|I^z|}]^\top\in\mathbb{R}^{|I^z|}$.  Then for
    \begin{equation*}
        \mathrm{diag}(u) = \begin{bmatrix}
            u_1 & \cdots & 0\\
            \vdots & \ddots & \vdots\\
            0 & \cdots & u_n
        \end{bmatrix}~~~\text{and}~~~
        M_{ij}=\begin{cases}
            -\lambda_{ij}, &~\text{if}~(i,j)\notin I^z,\\
            w_k~\text{that corresponds to $S_{ij}$}, &~\text{if}~(i,j)\in I^z,
        \end{cases}
    \end{equation*}
    the dual problem of problem \eqref{pro:sdp_primal} is
    \begin{align}\label{pro:sdp_dual}\tag{D}
    \begin{split}
        \max_{u,w}~~~&m+n+\sum\limits_{\substack{1\le i\le n\\ 1\le j\le m}} M_{ij}-\sum\limits_{i=1}^n u_i + 1\\
        s.t.~~~& 
        \begin{bmatrix}
            I_n+\mathrm{diag}(u) & \frac{1}{2}M\\
            \frac{1}{2}M^\top & I_m
        \end{bmatrix} \succeq 0.
    \end{split}
    \end{align}

\subsection{Examples where the SDP approximation is accurate}\label{subsec:examples}
    In this subsection, we prove that for certain classes of inscribable polytopes, the solution to problem~\eqref{pro:sdp_ori} corresponds to a solution to problem \eqref{pro:ori} with rank $d+1$.  

    Without loss of generality, suppose that the polytope is inscribed in the unit sphere and $(V,H,S^{\mathrm{true}})$ is an inscription.  Then, a feasible point of problem \eqref{pro:sdp_primal} is
    \begin{align}\label{eqs:ABS}
    \begin{split}
        A^*&=V^\top V+\mymathbb{1}_{n\times n},\\
        S^*&=S^{\mathrm{true}}=\mymathbb{1}_{n\times m}-V^\top H,\\
        B^*&=H^\top H+\mymathbb{1}_{n\times n}.
    \end{split}
    \end{align}
    The objective function value of problem \eqref{pro:sdp_primal} at $(A^*,B^*,S^*)$ is
    \begin{equation*}
        f_p^*  = \langle A^*, I_n\rangle + \langle B^*, I_m\rangle - \sum_{(i,j)\notin I^z}\lambda_{ij}S^*_{ij} + 1 = 2n+m+\sum\limits_{i=1}^mh_i^\top h_i-\sum_{(i,j)\notin I^z}\lambda_{ij}S^{\mathrm{true}}_{ij} + 1.
    \end{equation*}
    We want to show that $(A^*,B^*,S^*)$ is a solution to problem \eqref{pro:sdp_primal}.  If this is true, then \begin{equation*}
        X^*=\begin{bmatrix}
            1 & \mymathbb{1}_n^\top & \mymathbb{1}_m^\top\\
            \mymathbb{1}_n & A^* & S^*\\
            \mymathbb{1}_m & (S^*)^\top & B^*
        \end{bmatrix}
    \end{equation*}
    is a solution to problem \eqref{pro:ori} with rank $d+1$.

    Now we consider the dual problem \eqref{pro:sdp_dual}.  According to Lemma \ref{lem:psdschur}, its semidefinite constraint is equivalent to
    \begin{equation}\label{eq:schur_feasible}
        I_n+\mathrm{diag}(u)-\frac{1}{4}MM^\top\succeq 0.
    \end{equation}
    If we can find $(u^*,w^*)$ that satisfy \eqref{eq:schur_feasible} and $f_p^*=f_d^*$, where $f_d^*$ is the objective function of problem~\eqref{pro:sdp_dual} evaluated at $(u^*,w^*)$, then the duality gap is zero, so $(A^*,B^*,S^*)$ is a solution to problem~\eqref{pro:sdp_primal}.  Note that if we set 
    \begin{align*}
            u_i^* &= \overline{u}, i=1,\ldots,n,\\
    		w_i^* &= \overline{w}, i=1,\ldots,\left|I^z\right|,\\
            \lambda_{ij} &= \overline{\lambda}, (i,j)\notin I^z,
    \end{align*}
    then the two conditions are simplified to
    \begin{align*}
    \begin{cases}
        \lambda_{\max}(MM^\top) \leq 4 + 4\overline{u},\\
        n+\sum_{i=1}^m \|h_i\|^2-\overline{\lambda} \sum_{(i,j)\notin I^z} S_{ij}^{\mathrm{true}}=\overline{w}|I^z| -\overline{\lambda}(nm-|I^z|)-n\overline{u}.
    \end{cases}
    \end{align*}
    In the cases we are focusing on, our primal inscribed polytopes have the property of being facet transitive, i.e., there are rigid linear transformations that send the polytope to itself, and send any of its facets to any other of its facets.  Moreover, they are centered at the origin, and every facet has the same number, say $k$, of vertices. 
    
    Under these assumptions, the norms $\|h_i\|$ (which are the norms of the vertices of the polar polytope) are always the same, $\sum_{(i,j)\notin I^z} S_{ij}^{\mathrm{true}}=mn$ and $|I^z|=km$ so the two conditions are further simplified to 
    \begin{numcases}{}
        \lambda_{\max}(MM^\top) \leq 4 + 4\overline{u} \label{eq:simpl_dualfea},\\
        n(1+\overline{u}) +m\|h_1\|^2 =(\overline{\lambda}+\overline{w}) km \label{eq:simpl_0dualgap}.
    \end{numcases}

    In the remainder of this subsection, we are going to give several classes of polytopes, as examples, where we can prove that there exists a set of $(A^*,B^*,S^*)$, $(u^*,w^*)$, and $\lambda_{ij}$ such that the duality gap is zero.

\subsubsection{Example 1: $n$-gons}\label{subsubsec:ngons}
    Consider an $n$-gon ($n\ge 3$) in $\mathbb{R}^2$.  We order the facets and vertices so that facet $i$ contains vertices $i$ and $i+1$, for $i=1,...,n-1$, and facet $n$ contains vertices $1$ and $n$.
    We set the weights 
    \begin{equation*}
        \lambda_{ij} = \overline{\lambda} = \frac{2}{n\cos^2\frac{\pi}{n}}, (i,j)\notin I^z.
    \end{equation*}
    
    \begin{lemma}
        For the given polygon and weights, the pair $(u^*,w^*)$ given by
            \begin{equation*}
    		u_i^* = \overline{u}= \tan^2\frac{\pi}{n}, i=1,\ldots,n,~~~\text{and}~~~w_j^* = \overline{w}=\frac{n-2}{n\cos^2\frac{\pi}{n}}, j=1,\ldots,\left|I^z\right|,
    	\end{equation*}
        is feasible for problem \eqref{pro:sdp_dual}.
    \end{lemma}
    \begin{proof}
        For the $w^*$ given above, $M$ is the circulant matrix whose first row is
        $$\frac{1}{n\cos^2\frac{\pi}{n}}\left(n-2, -2, -2, \cdots , -2, n-2 \right),$$
        so $\lambda_{\max}(MM^\top)=4\sec^2\frac{\pi}{n}$ (see Appendix \ref{app:ngons_eigv} for details). Therefore, we have
        $$4+4\overline{u}= 4+4\tan^2\frac{\pi}{n} = 4\sec^2\frac{\pi}{n} = \lambda_{\max}(MM^\top)  $$
        which as we saw in \eqref{eq:simpl_dualfea} implies dual feasibility.

    $\hfill\qed$
    \end{proof}

    For our primal feasible point, we choose the one corresponding to the regular $n$-gon inscribed in the unit circle. The polar polytope of this is still a regular $n$-gon, circumscribed to the same circle, whose vertices can be checked to have $\|h_i\|^2=\sec^2\frac{\pi}{n}$. Moreover, there are two vertices per facet $(k=2)$ and it is a straightforward computation to check that \eqref{eq:simpl_0dualgap} holds.  We have proven the following result.
    \begin{theorem}
        For all $n \ge 3$ and the choice of weights $\lambda_{ij} = \overline{\lambda} = \frac{2}{n}\sec^2\frac{\pi}{n}, (i,j)\notin I^z$, problem~\eqref{pro:sdp_ori} has an optimal solution of rank $3$ that certifies inscribability of the $n$-gon.
    \end{theorem}

\subsubsection{Example 2: Simplices}\label{subsubsec:simplices}
    Consider a simplex in $\mathbb{R}^d$ with $d\ge 2$.  We order the vertices and facets such that vertex $i$ is in facet $i$, and set the weights 
    \begin{equation*}
        \lambda_{ij} = \overline{\lambda}= \frac{2d^2}{d+1},(i,j)\notin I^z.
    \end{equation*}
    
    \begin{lemma}
        For the given polytope and weights, the pair $(u^*,w^*)$ given by
            \begin{equation*}
            u_i^*  = \overline{u}= -1+d^2, i=1,\ldots,n,~~~\text{and}~~~w_i^* = \overline{w}= \frac{2d}{d+1}, i=1,\ldots,\left|I^z\right|,
    	\end{equation*}
        is feasible for problem \eqref{pro:sdp_dual}.
    \end{lemma}
    \begin{proof}
        For the $w^*$ given above, $M$ is the circulant matrix whose first row is $$\frac{2d}{d+1}\left(-d, 1, 1, \cdots , 1, 1 \right),$$ so $\lambda_{\max}(MM^\top)=4d^2$ (see Appendix \ref{app:simplices_eigv} for details).  Therefore, we have $$4+4\overline{u}= 4d^2 = \lambda_{\max}(MM^\top)$$ which as we saw in \eqref{eq:simpl_dualfea} implies dual feasibility.

    $\hfill\qed$
    \end{proof}
    
    For our primal feasible point, we pick a regular simplex centered at the origin. Again, the polar polytope of this is still a simplex centered at the origin, but with the norm of each vertex verifying
    $\|h_i\|^2=d^2$. There are now $k=d$ points per facet and it is once again a rote verification to see that~\eqref{eq:simpl_0dualgap} holds.  We have proven the following result.
    \begin{theorem}
        For all $d \geq 2$ and the choice of weights $\lambda_{ij} = \overline{\lambda}= \frac{2d^2}{d+1}, (i,j)\notin I^z$, problem~\eqref{pro:sdp_ori} has an optimal solution of rank $d+1$ that certifies inscribability of the $d$-simplex.
    \end{theorem}

\subsubsection{Example 3: Cubes}\label{subsubsec:cubes}
    Consider a cube in $\mathbb{R}^d$ with $d\ge 2$.  Set the weights
    \begin{equation*}
        \lambda_{ij} = \overline{\lambda}= d2^{1-d}, (i,j)\notin I^z.
    \end{equation*}
    
   \begin{lemma}\label{lem:cube}
        For the given polytope and weights, the pair $(u^*,w^*)$ given by
        \begin{equation*}
            u_i^*  = -1 + d^22^{1-d}, i=1,\ldots,n,~~~\text{and}~~~w_i^* = d2^{1-d}, i=1,\ldots,\left|I^z\right|,	
    	\end{equation*}
        is feasible for problem \eqref{pro:sdp_dual}.
    \end{lemma}
    \begin{proof}
        For the $w^*$ given above, $M$ is the matrix whose $2^n$ columns are all possible $\pm 1$ vectors of length $n$, multiplied by $d2^{1-d}$.  For this matrix $\lambda_{\max}(MM^\top)=\left(d2^{1-d}\right)^22^{d+1} = d^22^{3-d}$ (see Appendix \ref{app:cubescrosspolytopes_eigv} for details).  Therefore, we have $$4+4\overline{u}= 4d^22^{1-d} =  \lambda_{\max}(MM^\top)$$ which as we saw in \eqref{eq:simpl_dualfea}, implies dual feasibility.

    $\hfill\qed$
    \end{proof}

    For our primal feasible point, we pick the inscribed regular cube centered at the origin.  This time the polar polytope of this is the cross-polytope with vertices $\pm \sqrt{d} e_i^{(d)}, i=1,...,d$, so $\|h_i\|^2=d$ and there are $k=2^{d-1}$ points per facet.  Once again, it is a simple exercise to verify \eqref{eq:simpl_0dualgap}.  We have proven the following result.
    \begin{theorem}
        For all $d \ge 2$ and the choice of weights $\lambda_{ij} =  \overline{\lambda}= d2^{1-d}, (i,j)\notin I^z$, problem~\eqref{pro:sdp_ori} has an optimal solution of rank $d+1$ that certifies inscribability of the $d$-cube.
    \end{theorem}

\subsubsection{Example 4: Cross-polytopes}\label{subsubsec:crosspolytopes}
    The last case we will consider is a $d$-dimensional cross-polytope with $d\ge 2$. This is very similar to the previous example. We set the weights 
    \begin{equation*}
        \lambda_{ij} = \overline{\lambda}= 1.
    \end{equation*}
    
    \begin{lemma}
        For the given polytope and weights, the pair $(u^*,w^*)$ given by
        \begin{equation*}
            u_i^* = -1 + 2^{d-1}, i=1,\ldots,n,~~~\text{and}~~~w_i^* = 1, i=1,\ldots,\left|I^z\right|,	
    	\end{equation*}
        is feasible for problem \eqref{pro:sdp_dual}.
    \end{lemma}
    \begin{proof}
        Notice that $M$ is the transpose of the matrix $M$ in the proof of Lemma \ref{lem:cube} divided by $d2^{1-d}$, we have $\lambda_{\max}(MM^\top)=2^{d+1}$.  Therefore, \eqref{eq:simpl_dualfea} is still verified and implies dual feasibility.

    $\hfill\qed$
    \end{proof}
    
    Now we pick our primal feasible point the canonical cross-polytope with vertices $\pm e_i^{(d)}$.  Its polar polytope is the cube with vertices with coordinates $\pm 1$, so $\|h_i\|^2=d$ and there are $k=d$ vertices per facet.  We can check that \eqref{eq:simpl_0dualgap} is verified, proving the following result.  
    \begin{theorem}
        For all $d \ge 2$ and the choice of weights $\lambda_{ij} =  \overline{\lambda}= 1, (i,j)\notin I^z$, problem~\eqref{pro:sdp_ori} has an optimal solution of rank $d+1$ that certifies inscribability of the $d$-cross-polytope.
    \end{theorem}

\section{Algorithms}\label{sec:algs}
\subsection{SDP}
    The first algorithm to solve the inscribability problem is to solve the SDP problem~\eqref{pro:sdp_ori} and check if the solution gives an inscription.  In particular, to solve the SDP problem~\eqref{pro:sdp_ori}, we use CVX, which is a MATLAB package for solving convex optimization problems \cite{grant2008graph,grant2014cvx}.
    
    Clearly, the weights $\lambda_{ij}$ influence the solution.  We will provide and compare three different methods of tuning $\lambda_{ij}$ in our numerical experiments (see Section \ref{sec:numexp} for details).  We note that even if the SDP fails to provide a correct rank solution, its solution may still correspond to an inscription.  The reason is twofold.  First, rank computations are subject to numerical errors and thus the computed rank may not be accurate.  Second, the solution of SDP may correspond to a realization that is very close to an inscription.  This is especially possible for simplicial polytopes, as their combinatorial types are stable under small perturbations of vertices.

    Therefore, in the numerical experiments, we first solve the SDP problem, then, regardless of the rank of solutions, we extract the $V$ matrix from the solution (according to the proof of Theorem~\ref{thm:inscri}), scale the vertices to be on the sphere, and check if this is an inscription.

\subsection{Alternating projection}
    Another way to solve problem \eqref{pro:ori} is via the alternating projection (AP) algorithm between the constraint set of problem \eqref{pro:ori} and the set of matrices of rank $d+1$ \cite{fazel2004rank}.  For two closed convex sets with a nonempty intersection, the AP algorithm is known to converge globally to a point in the intersection \cite{bauschke2023introduction}.  However, in our cases the set of matrices of rank $d+1$ is non-convex and thus the algorithm only has local convergence \cite{bauschke2013restrictedb,bauschke2013restricteda,lewis2009local,noll2016local}.

    Let $X\in\mathbb{R}^{(n+m+1)\times(n+m+1)}$.  The projection of $X$ onto the constraint set of problem \eqref{pro:ori}, denoted by $P_\Omega(X)$, can be computed by solving the following convex optimization problem
    \begin{align}\label{pro:ProjonConstSet}
        \begin{split}
            \min_{P_\Omega(X)}~~~&\left\|P_\Omega(X)-X\right\|\\
            s.t.~~~& P_\Omega(X) = 
            \begin{bmatrix}
                1 & \mymathbb{1}_n^\top & \mymathbb{1}_m^\top\\
                \mymathbb{1}_n & A & S\\
                \mymathbb{1}_m & S^\top & B
            \end{bmatrix} \succeq 0\\
            & S_{ij} = 0,~\text{if}~(i,j)\in I^z\\
            & S_{ij} > 0,~\text{if}~(i,j)\notin I^z\\
            & A_{ii} = 2, i = 1,\ldots,n.
        \end{split}
    \end{align}
    In particular, we use CVX \cite{grant2008graph,grant2014cvx} to solve problem \eqref{pro:ProjonConstSet}.
    
    The projection of $X$ onto the set of matrices of rank $d+1$, denoted by $P_R(X)$ can be computed by singular value decomposition (SVD) \cite{eckart1936approximation}.  Suppose the SVD of $X$ is $$X=\sum\limits_{i=1}^{n+m+1}\sigma_iu_iv_i^\top$$ where $\sigma_1\ge\cdots\ge\sigma_{n+m+1}\ge 0$, then $$P_R(X)=\sum\limits_{i=1}^{d+1}\sigma_iu_iv_i^\top.$$

    The AP algorithm for solving problem \eqref{pro:ori} is described in Algorithm \ref{alg:ap}.
    \begin{algorithm}[!htb]\caption{Alternating projection algorithm for solving problem \eqref{pro:ori}}\label{alg:ap}
    \SetKwRepeat{Do}{do}{while}
    \DontPrintSemicolon
        \KwIn{Starting point $X\in\mathbb{R}^{(n+m+1)\times(n+m+1)}$; stopping tolerance $\epsilon$.}
        
        \Do{$E>\epsilon$}
        {  
            Compute the SVD $X=\sum\limits_{i=1}^{n+m+1}\sigma_iu_iv_i^\top$ where $\sigma_1\ge\cdots\ge\sigma_{n+m+1}\ge 0$.
            
            Set $Y=\sum\limits_{i=1}^{d+1}\sigma_iu_iv_i^\top$.
            
            Compute $P_\Omega(Y)$ by solving convex optimization problem \eqref{pro:ProjonConstSet}.

            Update $X=P_\Omega(Y)$ and $E=\|X-Y\|$.
        }
    \end{algorithm}

\subsection{Simplified alternating projection}
    Instead of computing the true projection of $X$ on $\Omega$, we can compute an approximation $\widetilde{P}_\Omega(X)$ by replacing specified elements with their fixed values, as shown in Algorithm \ref{alg:sap}.
    \begin{algorithm}[!htb]\caption{Simplified alternating projection algorithm for solving problem \eqref{pro:ori}}\label{alg:sap}
    \SetKwRepeat{Do}{do}{while}
    \DontPrintSemicolon
        \KwIn{Starting point $X\in\mathbb{R}^{(n+m+1)\times(n+m+1)}$; stopping tolerance $\epsilon$.}
        
        \Do{$E>\epsilon$}
        {  
            Compute the SVD $X=\sum\limits_{i=1}^{n+m+1}\sigma_iu_iv_i^\top$ where $\sigma_1\ge\cdots\ge\sigma_{n+m+1}\ge 0$.
            
            Set $Y=\sum\limits_{i=1}^{d+1}\sigma_iu_iv_i^\top$ and denote
            \begin{equation*}
                Y = \begin{bmatrix}
                    a & b_1^\top & b_2^\top\\
                    b_1 & Y_1 & Y_2\\
                    b_2 & Y_2^\top & Y_3
                \end{bmatrix}~\text{where}~a\in\mathbb{R},b_1\in\mathbb{R}^{n\times 1},b_2\in\mathbb{R}^{m\times 1}, Y_1\in\mathbb{R}^{n\times n}, Y_2\in\mathbb{R}^{n\times m}, Y_3\in\mathbb{R}^{m\times m}.
            \end{equation*}
            
            Compute $\widetilde{P}_\Omega(Y)$ by setting
            \begin{equation*}
                \widetilde{P}_\Omega(Y) = \begin{bmatrix}
                    1 & \mymathbb{1}_n^\top & \mymathbb{1}_m^\top\\
                    \mymathbb{1}_n & \widetilde{Y}_1 & \widetilde{Y}_2\\
                    \mymathbb{1}_m & \widetilde{Y}_2^\top & Y_3
                \end{bmatrix}
            \end{equation*}
            where
            \begin{equation*}
                \left(\widetilde{Y}_1\right)_{ij} = \begin{cases}
                    2, &~\text{if}~i=j,\\
                    \left(Y_1\right)_{ij}, &~\text{if otherwise},\\
                \end{cases}~~~\text{and}~~~
                \left(\widetilde{Y}_2\right)_{ij} = \begin{cases}
                    0, &~\text{if}~(i,j)\in I^z,\\
                    \left(Y_2\right)_{ij}, &~\text{if otherwise}.\\
                \end{cases}
            \end{equation*}

            Update $X=\widetilde{P}_\Omega(Y)$ and $E=\|X-Y\|$.
        }
    \end{algorithm}
    
    Clearly, this $\widetilde{P}_\Omega(X)$ may not be the same as the true projection, which will influence the accuracy of the algorithm.  However, as we do not need to compute the exact projection onto $\Omega$, this would speed up the algorithm and make it adaptable to more general settings (e.g., non-convex $\Omega$).  For simplicity, in the remainder of this paper, we denote this algorithm by SAP.

\section{Numerical experiments}\label{sec:numexp}
    In this section, we design numerical experiments to compare the performance of the three algorithms in Section \ref{sec:algs}.  In particular, we compare the accuracy and runtime of the three algorithms.  To check the accuracy of each algorithm, we first extract the vertices from the solution $X$ according to the algorithm given by the proof of Theorem \ref{thm:inscri}.  Then, we use polymake \cite{Gawrilow2000} to check if the vertices correspond to an inscription of the given polytope.
    
    The test problem sets we use are described in Subsection \ref{subsec:testsets}.  We provide methods to tune the weights $\lambda_{ij}$ in the SDP problem \eqref{pro:sdp_ori} in Subsection \ref{subsec:tunelmd}.  For the AP and SAP algorithms, we use the solution of the SDP problem \eqref{pro:sdp_ori} as the starting point.  The numerical results are presented in Subsection \ref{subsec:numres} and Appendix \ref{app:morenumres}.

\subsection{Test problems}\label{subsec:testsets}
    The numerical experiments are done on two test problem sets. In the first set, we randomly generate 100 inscribable polytopes with $n$ vertices in dimension $d$, where $8\le n\le 10$ and $5\le d\le 8$.  For each polytope, we uniformly generate $n$ points on the $(d-1)$-sphere, therefore the polytope is simplicial with probability one.  In the second test set, we use the list of simplicial 3-spheres with $n$ vertices provided by Firsching \cite{firsching2017realizability} (\url{https://firsching.ch/polytopes}).  The list consists of realizations of polytopal simplicial 3-spheres with $5\le n\le 10$, most of which are inscribed in the unit sphere. We only use the inscribed realizations for our experiments.

    In the first test set, the number of different combinatorial types for each parameter setting is shown in Table \ref{tab:numcombtype}.
    \begin{table}[!htb]
    \caption{Number of different combinatorial types in the first test set}
    \label{tab:numcombtype}
    \centering
    \begin{tabular}{ccccccccccc}
    \hline
    $n8d5$ & $n9d5$ & $n10d5$ & $n8d6$ & $n9d6$ & $n10d6$ & $n9d7$ & $n10d7$ & $n10d8$\\ \hline
    7 & 59 & 92 & 3 & 13 & 89 & 3 & 20 & 3\\
    \hline
    \end{tabular}
    \end{table}\FloatBarrier
    
    We note that although the numbers of different combinatorial types when $n=d+2$ and $n=d+3$ are relatively small in Table \ref{tab:numcombtype}, they cover most of the possible combinatorial types.  In fact, as mentioned in \cite{grunbaum2003convex}, when $n=d+2$, there exist $\floor*{\frac{d}{2}}$ different combinatorial types; when $n=d+3$, there exists 
    \begin{equation*}
        2^{\floor*{\frac{d}{2}}} - \floor*{\frac{d+4}{2}} + \frac{1}{4(d+3)}\sum\limits_{\text{$h$ odd divisor of $d+3$}}\varphi(h)2^{\frac{d+3}{h}}
    \end{equation*}
    different combinatorial types, where $\floor*{\cdot}$ is the floor function and $\varphi(\cdot)$ is the Euler's $\varphi$-function.  Therefore, when $n=8$ and $d=5$, there are 8 different combinatorial types; when $n=8$ and $d=6$, there are 3 different combinatorial types; when $n=9$ and $d=6$, there are 18 different combinatorial types; when $n=9$ and $d=7$, there are 3 different combinatorial types; when $n=10$ and $d=7$, there are 29 different combinatorial types; when $n=10$ and $d=8$, there are 4 different combinatorial types.

\subsection{Tuning $\lambda_{ij}$}\label{subsec:tunelmd}
    We provide and compare three methods to tune the weights $\lambda_{ij}$ in the SDP problem \eqref{pro:sdp_ori}.  We note that the third method uses the inscription information of a polytope while the other two methods do not need any inscription information.  

    The first method is based on observation of the examples in Subsection \ref{subsec:examples}, which suggests that a potential good choice of $\lambda_{ij}$ would be all $\lambda_{ij}=\frac{2d}{n}$.  

    In the second method, we tune $\lambda_{ij}$ according to a heuristic described in Algorithm~\ref{alg:tunelmdheuristic}.  In each iteration, we extract the vertices from the SDP solution and check if the structure of the polytope constructed from these vertices matches the structure presented by the true slack matrix, i.e., if the zero pattern of the slack matrix corresponding to the SDP solution matches that of the true slack matrix.  If the structure matches, then an inscription is found and the algorithm stops; otherwise, for each facet that does not match, we increase all corresponding $\lambda_{ij}$ by multiplying them with a fixed rate $\lambda^{\mathrm{inc}}$.  We repeat this process until an inscription is found or the maximum $\lambda_{ij}$ is bigger than $\lambda^{\max}$.  In our tests, we use Algorithm~\ref{alg:tunelmdheuristic} with all $\lambda_{ij}^{\mathrm{init}}=\frac{2d}{n}$, $\lambda^{\mathrm{inc}}=\frac{n}{d}$, and $\lambda^{\max}=\frac{2d}{n}(\frac{n}{d})^{10}$, i.e., we stop if an inscription is found or there is at least one facet that does not match for more than 10 iterations.
    \begin{algorithm}[!htb]\caption{Solving the SDP problem \eqref{pro:sdp_ori} with a heuristic of tuning $\lambda_{ij}$}\label{alg:tunelmdheuristic}
        \KwIn{Polytope $P$ with slack matrix $S_P$; initial weights $\lambda_{ij}^{\mathrm{init}}$; weights increase rate $\lambda^{\mathrm{inc}}$; maximum weight $\lambda^{\max}$.}
        Set $\lambda_{ij}=\lambda_{ij}^{\mathrm{init}}$ for $i=1,\ldots,n$ and $j=1,\ldots,m$.
        
        \While{$\max\{\lambda_{ij}:i=1,\ldots,n,j=1,\ldots,m\}\le \lambda^{\max}$}
        {
            Solve the SDP problem \eqref{pro:sdp_ori} with current $\lambda_{ij}$ and denote the solution by $X$.
            
            Extract the $V=[v_1\cdots v_n]$ matrix from $X$ and normalize $v_1,\ldots,v_n$.
            
             
            Compare the structure $\mathrm{conv}(v_1,\ldots,v_n)$ to the structure presented by slack matrix $S_P$ and denote by $F^{\mathrm{bad}}$ the set of indices of facets that do not match.
                     
            \uIf{$F^{\mathrm{bad}}=\varnothing$}
            {
                Return $\mathrm{conv}(v_1,\ldots,v_n)$ as an inscription of $P$
            }
            \Else
            {
                Set $\lambda_{ij}=\lambda^{\mathrm{inc}}\lambda_{ij}$ for each $(i,j)\in \{1,\ldots,n\}\times F^{\mathrm{bad}}$.
            }
        }
        Return $\mathrm{conv}(v_1,\ldots,v_n)$.
    \end{algorithm}

    For the third method, as all the test problems we used are inscribable and we already know an inscription, we can tune the weights $\lambda_{ij}$ by minimizing the duality gap of the SDP problem~\eqref{pro:sdp_ori} for the given inscription.  Thus, if it is possible to find an inscription as a solution to problem \eqref{pro:sdp_ori}, then this method would find the weights to produce this solution.  However, to construct the slack matrix from the inscription, we need to transform the polytope such that it contains the origin in the interior and all vertices are on the unit sphere. 
    
    According to Lemma \ref{lem:0ininterior}, we preprocess the inscription by Algorithm~\ref{alg:preprocesspolys}.  Then, we tune $\lambda_{ij}$ such that the duality gap of the SDP problem~\eqref{pro:sdp_ori} is minimized for the preprocessed inscription.
    \begin{algorithm}[!htb]\caption{Preprocessing the inscription of polytopes}\label{alg:preprocesspolys}
    \SetKwRepeat{Do}{do}{while}
    \DontPrintSemicolon
        \KwIn{Vertices of the inscribed polytope $\{v_1,\ldots,v_n\}$.}
        Compute the center of the polytope $v_c=\sum\limits_{i=1}^n v_i\slash n$.

        Find the rotation matrix $A$ such that $Av_c=[\|v_c\|,0,\ldots,0]^\top$.

        \For{$i=1,\ldots,n$}
        {
            Rotate $v_i$ and denote $\widetilde{v}_i=Av_i$.
            
            Denote $\widetilde{v}_i=[\widetilde{v}_i^1,\ldots,\widetilde{v}_i^d]$.

            Apply the projective transformation to $\widetilde{v}_i$ and get
            \begin{equation*}
                w_i=\left[w_i^1,\ldots,w_i^d\right],~~\text{where}~~
                w_i^j=\begin{cases}
                    \frac{\widetilde{v}_i^1-\|v_c\|}{1-\|v_c\|\widetilde{v}_i^1}, &\text{if $j=1$},\\
                    \frac{\sqrt{1-\|v_c\|^2}\widetilde{v}_i^j}{1-\|v_c\|\widetilde{v}_i^1}, &\text{if $j>1$}.
                \end{cases}
            \end{equation*}

            Return $\{w_1,\ldots,w_n\}$ as the new vertices of the inscribed polytope.
        }
    \end{algorithm}\FloatBarrier

\subsection{Numerical results}\label{subsec:numres}
    In this subsection, we present the results of the numerical experiment.  For simplicity, we denote the three methods of tuning $\lambda_{ij}$ by the following:
    \begin{itemize}
        \item $\lambda^c$: tuning $\lambda_{ij}$ by setting all $\lambda_{ij}=\frac{2d}{n}$;
        \item $\lambda^h$: tuning $\lambda_{ij}$ by Algorithm \ref{alg:tunelmdheuristic};
        \item $\lambda^*$: tuning $\lambda_{ij}$ by minimizing the duality gap at the given inscription.
    \end{itemize}
    We make use of the notations SDP-$\lambda$, SAP-$\lambda$, and AP-$\lambda$, where $\lambda\in\{\lambda^c,\lambda^h,\lambda^*\}$, to denote algorithms.  The SDP-$\lambda$ notation means solving the SDP problem \eqref{pro:sdp_ori} with $\lambda_{ij}$ tuned by the method determined by $\lambda$.  The SAP-$\lambda$ notation means the SAP algorithm with the starting point set to be the solution of SDP-$\lambda$.  The AP-$\lambda$ notation means the AP algorithm with the starting point set to be the solution of SDP-$\lambda$.  Under this setting, if the solution of SDP-$\lambda$ already gives an inscription, then SAP-$\lambda$ and AP-$\lambda$ will not move away from the starting point and thus the accuracy of SAP-$\lambda$ and AP-$\lambda$ are at least as good as SDP-$\lambda$.

    We note that algorithms based on $\lambda^*$ are not implementable for polytopes without a known inscription.  We only use their results as a reference to demonstrate the effectiveness and efficiency of algorithms based on the other two tuning methods. 

\subsubsection{Results on random polytopes}
    Figure \ref{fig:randacc} and Table \ref{tab:randtime} present the accuracy and runtime results on the first test set.  The complete data corresponding to Figure \ref{fig:randacc} can be found in Appendix \ref{app:morenumres}.

    In Figure \ref{fig:randacc}, we use three different colors to represent the three different ways of tuning $\lambda_{ij}$ described in Subsection \ref{subsec:tunelmd}.  We use plain colored bars to present the results of SDP-$\lambda$.  For SAP-$\lambda$ and AP-$\lambda$, as their accuracy is at least as good as SDP-$\lambda$, we stack their results onto the bars of SDP-$\lambda$ and use patterned bars to represent their improvements over SDP-$\lambda$.
    
    From Figure \ref{fig:randacc} we can see that SDP-$\lambda^*$ and SDP-$\lambda^h$ have competitive accuracy and are clearly more accurate than SDP-$\lambda^c$.  The SAP algorithm has an obvious accuracy improvement on SDP-$\lambda^*$ and SDP-$\lambda^c$, while having no improvements on SDP-$\lambda^h$.  These demonstrate the effectiveness of Algorithm \ref{alg:tunelmdheuristic}.
    
    The AP algorithm improves the accuracy of SDP-$\lambda^*$ even more than the SAP algorithm.  However, as shown in Table~\ref{tab:randtime}, the AP algorithm takes a much longer time than other algorithms so we only run AP-$\lambda^*$ to reveal the potential of the AP algorithm in terms of accuracy.  Future research may explore this potential and improve the efficiency of the AP algorithm.
    \begin{figure}[!htb]
         \centering
         \includegraphics[width=1\textwidth]{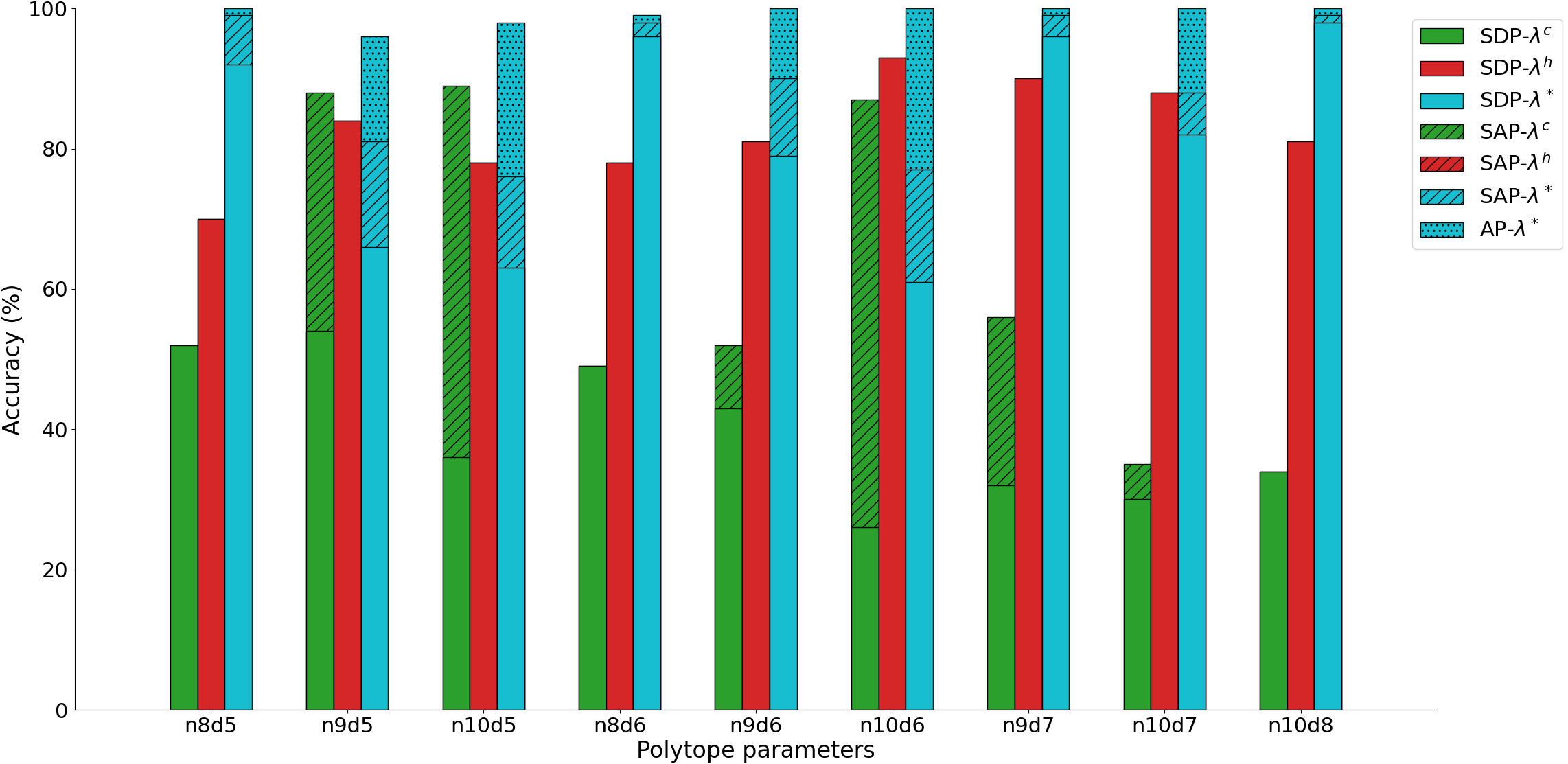}
         \caption{Accuracy on random polytopes}
         \label{fig:randacc}
    \end{figure}\FloatBarrier

    In Table \ref{tab:randtime}, the numbers without brackets are the average runtime per polytope in seconds, while the numbers inside the brackets are the maximum runtime required to solve all tests with the corresponding parameter setting.  In general, the SDP algorithm is the most efficient algorithm.  Besides, the maximum runtime of the SDP algorithm is similar to the average runtime in most cases, which implies that the average runtime of the SDP algorithm shown in the table is a good estimate of its general performance.  Therefore, Table \ref{tab:randtime} demonstrates the robustness and efficiency of the SDP algorithm.
    
    For the SAP and AP algorithms, since we use the solution of the SDP algorithm as the starting point, the runtime required to solve one test ranges from the runtime of the SDP algorithm (occurs when the SDP solution already satisfies the stopping condition) to infinity (occurs when the algorithms never converge).  Therefore, we declare that the SAP and AP algorithms failed if they required more than 10 hours to solve one test.  In our experiments, we only observed two failed tests, both happened while using the AP-$\lambda^*$ algorithm: one occurred when $n=9$ and $d=5$, and the other occurred when $n=10$ and $d=5$.  For these two failed tests, we do not count their runtime data while computing the average and maximum runtime, instead, we add a ``$\ge$" symbol before the computed values in the table.  Table \ref{tab:randtime} presents the average and maximum runtime of all successful tests under each parameter setting.  Clearly, the SDP and SAP algorithms are much more efficient than the AP algorithm.

    In summary, Figure \ref{fig:randacc} and Table \ref{tab:randtime} imply that the SDP algorithm is the most efficient and robust.  Although the AP algorithm is the most accurate, it takes much more runtime than the other two.  Besides, the heuristic of tuning $\lambda_{ij}$ (Algorithm \ref{alg:tunelmdheuristic}) is very effective and makes the accuracy of SDP-$\lambda^h$ comparable with SDP-$\lambda^*$, while not requiring a known inscription to implement. 
    
    \begin{table}[!htb]
    \caption{Average runtime per polytope (seconds)}
    \label{tab:randtime}
    \centering
    \begin{tabular}{cccccccccccc}
    \hline
    Alg$\backslash$Para & $n8d5$ & $n9d5$ & $n10d5$ & $n8d6$ & $n9d6$ & $n10d6$ & $n9d7$ & $n10d7$ & $n10d8$\\ \hline
    \multirow{2}{*}{SDP-$\lambda^c$\vspace{0.1cm}} & \multirow{2}{*}{0.40\vspace{0.1cm}} & \multirow{2}{*}{0.53\vspace{0.1cm}} & \multirow{2}{*}{0.73\vspace{0.1cm}} & \multirow{2}{*}{0.36\vspace{0.1cm}} & \multirow{2}{*}{0.49\vspace{0.1cm}} & \multirow{2}{*}{0.90\vspace{0.1cm}} & \multirow{2}{*}{0.39\vspace{0.1cm}} & \multirow{2}{*}{0.66\vspace{0.1cm}} & \multirow{2}{*}{0.47\vspace{0.1cm}}\\
    {\ssmall (max)} & {\ssmall (0.50)} & {\ssmall (0.75)} & {\ssmall (1.01)} & {\ssmall (0.46)} & {\ssmall (0.90)} & {\ssmall (1.18)} & {\ssmall (0.49)} & {\ssmall (1.26)} & {\ssmall (0.85)}\\
    \multirow{2}{*}{SAP-$\lambda^c$} & \multirow{2}{*}{4.98\vspace{0.1cm}} & \multirow{2}{*}{3.43\vspace{0.1cm}} & \multirow{2}{*}{15.13\vspace{0.1cm}} & \multirow{2}{*}{19.62\vspace{0.1cm}} & \multirow{2}{*}{21.01\vspace{0.1cm}} & \multirow{2}{*}{24.61\vspace{0.1cm}} & \multirow{2}{*}{37.76\vspace{0.1cm}} & \multirow{2}{*}{86.86\vspace{0.1cm}} & \multirow{2}{*}{37.52\vspace{0.1cm}}\\
    {\ssmall (max)} & {\ssmall (25.80)} & {\ssmall (63.73)} & {\ssmall (158.37)} & {\ssmall (47.93)} & {\ssmall (74.00)} & {\ssmall (189.07)} & {\ssmall (47.09)} & {\ssmall (170.55)} & {\ssmall (69.12)}\\
    \multirow{2}{*}{SDP-$\lambda^h$} & \multirow{2}{*}{0.59\vspace{0.1cm}} & \multirow{2}{*}{1.58\vspace{0.1cm}} & \multirow{2}{*}{5.10\vspace{0.1cm}} & \multirow{2}{*}{3.29\vspace{0.1cm}} & \multirow{2}{*}{1.19\vspace{0.1cm}} & \multirow{2}{*}{2.89\vspace{0.1cm}} & \multirow{2}{*}{4.41\vspace{0.1cm}} & \multirow{2}{*}{2.32\vspace{0.1cm}} & \multirow{2}{*}{5.17\vspace{0.1cm}}\\
    {\ssmall (max)} & {\ssmall (3.93)} & {\ssmall (12.32)} & {\ssmall (39.07)} & {\ssmall (13.76)} & {\ssmall (3.95)} & {\ssmall (19.65)} & {\ssmall (8.21)} & {\ssmall (6.21)} & {\ssmall (9.78)}\\
    \multirow{2}{*}{SAP-$\lambda^h$} & \multirow{2}{*}{0.60\vspace{0.1cm}} & \multirow{2}{*}{5.45\vspace{0.1cm}} & \multirow{2}{*}{35.75\vspace{0.1cm}} & \multirow{2}{*}{5.50\vspace{0.1cm}} & \multirow{2}{*}{1.24\vspace{0.1cm}} & \multirow{2}{*}{12.34\vspace{0.1cm}} & \multirow{2}{*}{6.06\vspace{0.1cm}} & \multirow{2}{*}{2.46\vspace{0.1cm}} & \multirow{2}{*}{12.45\vspace{0.1cm}}\\
    {\ssmall (max)} & {\ssmall (3.93)} & {\ssmall (78.71)} & {\ssmall (201.05)} & {\ssmall (39.27)} & {\ssmall (3.95)} & {\ssmall (202.04)} & {\ssmall (44.79)} & {\ssmall (14.01)} & {\ssmall (69.51)}\\
    \multirow{2}{*}{SDP-$\lambda^*$} & \multirow{2}{*}{0.81\vspace{0.1cm}} & \multirow{2}{*}{1.08\vspace{0.1cm}} & \multirow{2}{*}{1.42\vspace{0.1cm}} & \multirow{2}{*}{0.67\vspace{0.1cm}} & \multirow{2}{*}{0.91\vspace{0.1cm}} & \multirow{2}{*}{1.43\vspace{0.1cm}} & \multirow{2}{*}{1.18\vspace{0.1cm}} & \multirow{2}{*}{1.68\vspace{0.1cm}} & \multirow{2}{*}{1.16\vspace{0.1cm}}\\
    {\ssmall (max)} & {\ssmall (0.97)} & {\ssmall (1.62)} & {\ssmall (1.69)} & {\ssmall (0.87)} & {\ssmall (1.05)} & {\ssmall (1.59)} & {\ssmall (1.41)} & {\ssmall (2.23)} & {\ssmall (1.70)}\\
    \multirow{2}{*}{SAP-$\lambda^*$} & \multirow{2}{*}{1.40\vspace{0.1cm}} & \multirow{2}{*}{11.83\vspace{0.1cm}} & \multirow{2}{*}{34.42\vspace{0.1cm}} & \multirow{2}{*}{1.14\vspace{0.1cm}} & \multirow{2}{*}{6.62\vspace{0.1cm}} & \multirow{2}{*}{42.09\vspace{0.1cm}} & \multirow{2}{*}{1.43\vspace{0.1cm}} & \multirow{2}{*}{14.47\vspace{0.1cm}} & \multirow{2}{*}{1.79\vspace{0.1cm}}\\
    {\ssmall (max)} & {\ssmall (32.33)} & {\ssmall (66.04)} & {\ssmall (162.21)} & {\ssmall (29.24)} & {\ssmall (62.90)} & {\ssmall (197.11)} & {\ssmall (13.30)} & {\ssmall (165.86)} & {\ssmall (61.28)}\\
    \multirow{2}{*}{AP-$\lambda^*$} & \multirow{2}{*}{4.49\vspace{0.1cm}} & \multirow{2}{*}{$\ge$117.21\vspace{0.1cm}} & \multirow{2}{*}{$\ge$567.20\vspace{0.1cm}} & \multirow{2}{*}{29.15\vspace{0.1cm}} & \multirow{2}{*}{46.25\vspace{0.1cm}} & \multirow{2}{*}{280.89\vspace{0.1cm}} & \multirow{2}{*}{3.98\vspace{0.1cm}} & \multirow{2}{*}{83.78\vspace{0.1cm}} & \multirow{2}{*}{4.94\vspace{0.1cm}}\\
    {\ssmall (max)}  & {\ssmall (111.83)} & {\ssmall ($\ge$2285.51)} & {\ssmall ($\ge$7975.37)} & {\ssmall (2632.58)} & {\ssmall (1238.94)} & {\ssmall (2554.53)} & {\ssmall (119.55)} & {\ssmall (1258.57)} & {\ssmall (348.78)}\\
    \hline
    \end{tabular}
    \end{table}\FloatBarrier

\subsubsection{Results on the list of inscribable polytopes}
    Figure \ref{fig:listacc} and Table \ref{tab:listtime} present the accuracy and runtime results on the second test set.  The complete data corresponding to Figure \ref{fig:listacc} can be found in Appendix \ref{app:morenumres}.  Note that for $n=10$, we only run 100 tests from the list provided by Firsching \cite{firsching2017realizability}.  We also note that we do not run SDP-$\lambda^c$ and SAP-$\lambda^c$ on the second test set since they have clearly worse performance than other algorithms according to the results on the first test set (Figure \ref{fig:randacc}).

    Similar to the experiments on the first test set, we do not count the runtime data of the tests that require more than 10 hours to solve.  This time, we observed 4 failed tests: all occurred when $n=9$ while using the AP-$\lambda^*$ algorithm.  Figure \ref{fig:listacc} and Table \ref{tab:listtime} have similar patterns as Figure \ref{fig:randacc} and Table \ref{tab:randtime}.  The SAP algorithm improves the accuracy of SDP-$\lambda^*$ by an obvious amount while having no improvements on SDP-$\lambda^h$.  Although the AP algorithm is the most accurate, it is much slower than the other algorithms.  Besides, Table \ref{tab:listtime} again demonstrates the efficiency and robustness of the SDP algorithm.

    It is worth noticing that the accuracy of the SDP-$\lambda^h$ decreases as $n$ increases in this test set.  This implies that the heuristic we used to tune $\lambda_{ij}$ presented in Algorithm~\ref{alg:tunelmdheuristic} is not a universally good strategy.  Future research may explore better parameter settings for Algorithm~\ref{alg:tunelmdheuristic} or other more adaptive ways of tuning $\lambda_{ij}$.
    \begin{figure}[!htb]
         \centering
         \includegraphics[width=1\textwidth]{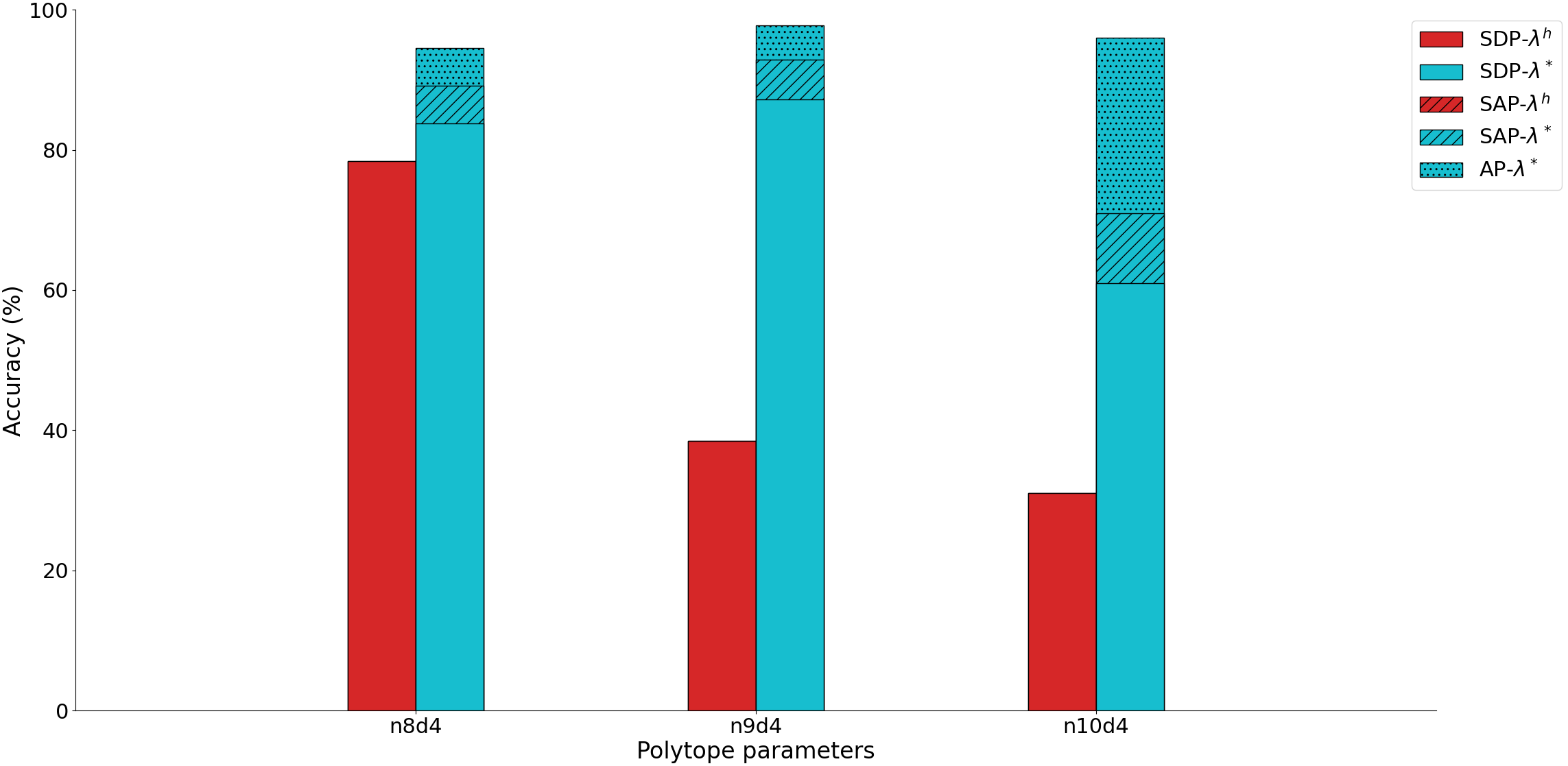}
         \caption{Accuracy on random polytopes}
         \label{fig:listacc}
    \end{figure}\FloatBarrier

    \begin{table}[!htb]
    \caption{Average runtime per polytope (s)}
    \label{tab:listtime}
    \centering
    \begin{tabular}{ccccccc}
    \hline
    Alg$\backslash n$ & 5 & 6 & 7 & 8 & 9 & 10 \\ \hline
    \multirow{2}{*}{SDP-$\lambda^h$} & \multirow{2}{*}{0.27\vspace{0.1cm}} & \multirow{2}{*}{0.59\vspace{0.1cm}} & \multirow{2}{*}{0.40\vspace{0.1cm}} & \multirow{2}{*}{1.05\vspace{0.1cm}} & \multirow{2}{*}{5.92\vspace{0.1cm}} & \multirow{2}{*}{6.63\vspace{0.1cm}}\\
    {\ssmall (max)} & {\ssmall (0.27)} & {\ssmall (0.92)} & {\ssmall (0.81)} & {\ssmall (9.56)} & {\ssmall (16.40)} & {\ssmall (14.43)}\\
    \multirow{2}{*}{SAP-$\lambda^h$}  & \multirow{2}{*}{0.34\vspace{0.1cm}} & \multirow{2}{*}{0.60\vspace{0.1cm}} & \multirow{2}{*}{0.42\vspace{0.1cm}} & \multirow{2}{*}{3.54\vspace{0.1cm}} & \multirow{2}{*}{37.09\vspace{0.1cm}} & \multirow{2}{*}{33.30\vspace{0.1cm}} \\
    {\ssmall (max)} & {\ssmall (0.34)} & {\ssmall (0.93)} & {\ssmall (0.85)} & {\ssmall (40.18)} & {\ssmall (256.38)} & {\ssmall (78.73)}\\
    \multirow{2}{*}{SDP-$\lambda^*$} & \multirow{2}{*}{0.41\vspace{0.1cm}} & \multirow{2}{*}{0.49\vspace{0.1cm}} & \multirow{2}{*}{0.61\vspace{0.1cm}} & \multirow{2}{*}{0.91\vspace{0.1cm}} & \multirow{2}{*}{1.29\vspace{0.1cm}} & \multirow{2}{*}{1.46\vspace{0.1cm}}\\
    {\ssmall (max)} & {\ssmall (0.41)} & {\ssmall (0.50)} & {\ssmall (0.67)} & {\ssmall (1.48)} & {\ssmall (1.86)} & {\ssmall (2.02)}\\
    \multirow{2}{*}{SAP-$\lambda^*$}  & \multirow{2}{*}{0.42\vspace{0.1cm}} & \multirow{2}{*}{0.50\vspace{0.1cm}} & \multirow{2}{*}{5.11\vspace{0.1cm}} & \multirow{2}{*}{4.36\vspace{0.1cm}} & \multirow{2}{*}{4.95\vspace{0.1cm}} & \multirow{2}{*}{18.92\vspace{0.1cm}} \\
    {\ssmall (max)} & {\ssmall (0.42)} & {\ssmall (0.50)} & {\ssmall (23.04)} & {\ssmall (33.63)} & {\ssmall (82.63)} & {\ssmall (76.91)}\\
    \multirow{2}{*}{AP-$\lambda^*$}  & \multirow{2}{*}{0.42\vspace{0.1cm}} & \multirow{2}{*}{0.50\vspace{0.1cm}} & \multirow{2}{*}{15.08\vspace{0.1cm}} & \multirow{2}{*}{86.73\vspace{0.1cm}} & \multirow{2}{*}{$\ge$105.02\vspace{0.1cm}} & \multirow{2}{*}{317.72\vspace{0.1cm}}\\ 
    {\ssmall (max)} & {\ssmall (0.42)} & {\ssmall (0.51)} & {\ssmall (72.89)} & {\ssmall (1582.87)} & {\ssmall ($\ge$8565.24)} & {\ssmall (9769.08)}\\
    \hline
    \end{tabular}
    \end{table}\FloatBarrier



\section{Conclusion}\label{sec:conclusion}
    This paper provided a new necessary and sufficient condition for the inscribability of a polytope.  Using slack matrices, we characterized the problem of determining the inscribability of a polytope as an equivalent minimum rank optimization problem.  We proposed an SDP approximation for the minimum rank optimization problem and proved that it is tight for certain classes of polytopes.  For general polytopes, we provided three algorithms with numerical comparisons for the inscribability problem.  Although the optimization problems and algorithms we proposed depend on the number of vertices and facets, they are independent of the dimension of polytopes, which is an advantage over the method proposed by Firsching \cite{firsching2017realizability}.  Numerical results demonstrated our SDP approximation's efficiency, accuracy, and robustness for determining the inscribability of simplicial polytopes of dimensions $4\le d\le 8$ and vertices $n$ up to 10.  Therefore, the SDP approximation proposed in this paper can be used as a theoretical and practical tool for determining inscribability in high dimensions.  In particular, we recommend the following procedure for determining inscribability:
    \begin{itemize}
        \item {\bf Step 1:} Run SDP-$\lambda^h$ and check if the solution gives an inscription.  If yes, then the polytope is inscribable.
        \item {\bf Step 2:} Run SAP with the starting point set to be the solution of SDP-$\lambda^h$ and check if the solution gives an inscription.  If yes, then the polytope is inscribable.
        \item {\bf Step 3:} Run AP with the starting point set to be the solution of SDP-$\lambda^h$ and check if the solution gives an inscription.  If yes, then the polytope is inscribable.
    \end{itemize}

    However, we should note that all algorithms in this paper can not determine non-inscribability of polytopes.  In particular, if the algorithm gives a solution with rank $d+1$ to problem \eqref{pro:ori}, then we can conclude that the polytope is inscribable; if the algorithm cannot find a rank $d+1$ solution, then we cannot conclude that the polytope is non-inscribable.  An extension of this paper could be to derive general non-inscribability certificates for combinatorial polytopes.  The techniques proposed in \cite{gouveia2023general} to derive non-realizability could be used to help.

    Another future direction would be exploring more effective and efficient methods for solving the minimum rank optimization problem \eqref{pro:ori}.  As shown in our numerical experiments (Section \ref{sec:numexp}), the alternating projecting method has great potential in terms of accuracy but is rather inefficient.  Future work may focus on improving the efficiency of the alternating projecting method.

    It is also worth developing better heuristics for tuning $\lambda_{ij}$.  One natural extension is to design adaptive strategies for Algorithm \ref{alg:tunelmdheuristic}, instead of increasing $\lambda_{ij}$ by a fixed rate $\lambda^{\mathrm{inc}}$.  The weights tuning problem can also be viewed as a black-box optimization (BBO) or derivative-free optimization (DFO) problem, so many classic BBO or DFO methods can be applied.

\bibliographystyle{siam}
\bibliography{references}

\appendix

\section{Eigenvalue computations}
\subsection{Eigenvalues of $MM^\top$ in Subsubsection \ref{subsubsec:ngons}}\label{app:ngons_eigv}
Notice that 
\begin{equation*}
    MM^\top = \frac{(n-2)^2}{n^2\cos^4\frac{\pi}{n}} \begin{bmatrix}
        a & b & c & \cdots & b\\
        b & a & b & \cdots & c\\    
        c & b & a & \cdots & c\\
        \vdots & \vdots & \vdots & \ddots & \vdots\\
        b & c & c & \cdots & a
\end{bmatrix},
\end{equation*}
where
\begin{equation*}
    a=\frac{2n}{n-2},~~b=\frac{n(n-4)}{(n-2)^2},~~c=-\frac{4n}{(n-2)^2}.
\end{equation*}

We now consider two cases: $n=3$ and $n\ge 4$.  If $n=3$, then the eigenvalues of $MM^\top$ are
\begin{align*}
    \lambda_j &= \frac{1}{9\cos^4\frac{\pi}{3}}\left(a+be^{\frac{2\pi ji}{3}}+be^{\frac{4\pi ji}{3}}\right)\\
    &= 
    \begin{cases}
        \frac{1}{9\cos^4\frac{\pi}{3}}\left(a+2b\right)=0,&~\text{if $j=0$},\\
        \frac{1}{9\cos^4\frac{\pi}{3}}\left(a+2b\cos\frac{2\pi}{3}\right)=\frac{1}{\cos^4\frac{\pi}{3}},&~\text{if $j=1,2$}.
    \end{cases}
\end{align*}
Thus, the largest eigenvalue of $MM^\top$ is 
\begin{equation*}
    \lambda_{\max}(MM^\top) = \lambda_1=\lambda_2=\frac{1}{\cos^4\frac{\pi}{3}}=\frac{4}{\cos^2\frac{\pi}{n}}.
\end{equation*}

If $n\ge 4$, then the eigenvalues of $MM^\top$ are
\begin{align*}
    \lambda_j &= \frac{(n-2)^2}{n^2\cos^4\frac{\pi}{n}}\left(a+b\omega^j+c\sum\limits_{k=2}^{n-2}\omega^{kj}+b\omega^{(n-1)j}\right)\\
    &= 
    \begin{cases}
        \frac{(n-2)^2}{n^2\cos^4\frac{\pi}{n}}\left(a+2b+(n-3)c\right)=0,&~\text{if $j=0$},\\
        \frac{(n-2)^2}{n^2\cos^4\frac{\pi}{n}}\left(a+2b\cos\frac{2\pi j}{n}-c\left(1+2\cos\frac{2\pi j}{n}\right)\right),&~\text{if $j=1,\ldots,n-1$},
    \end{cases}
\end{align*}
where $\omega=e^{2\pi i\slash n}$.
Since $a>0$, $b\ge0$, and $c<0$, the largest eigenvalue of $MM^\top$ is
\begin{equation*}
    \lambda_{\max}(MM^\top) = \lambda_1=\lambda_{n-1}=\frac{(n-2)^2}{n^2\cos^4\frac{\pi}{n}}\left(a+2b\cos\frac{2\pi}{n}-c\left(1+2\cos\frac{2\pi}{n}\right)\right)=\frac{4}{\cos^2\frac{\pi}{n}}.
\end{equation*}

\subsection{Eigenvalues of $MM^\top$ in Subsubsection \ref{subsubsec:simplices}}\label{app:simplices_eigv}
Notice that
\begin{align*}
    MM^\top &= \left(\frac{2d}{d+1}\right)^2 \begin{bmatrix}
        d^2+d & -1-d & \cdots & -1-d\\
        -1-d & d^2+d & \cdots & -1-d\\
        \vdots & \vdots & \ddots & \vdots\\
        -1-d & -1-d & \cdots & d^2+d
    \end{bmatrix}\\
    &= \left(\frac{2d}{d+1}\right)^2\left(\left(d+1\right)^2I_n-\left(d+1\right)\mymathbb{1}_{n\times n}\right)\\
    &= 4d^2I_n-\frac{4d^2}{d+1}\mymathbb{1}_{n\times n}.
\end{align*}
Notice that the eigenvalues of an all-one matrix with dimension $n$ are $\{n,0\}$.  Therefore, the eigenvalues of $MM^\top$ are
\begin{equation*}
    \lambda\in\left\{4d^2, 0\right\}
\end{equation*}
and so
\begin{equation*}
    \lambda_{\max}(MM^\top) = 4d^2.
\end{equation*}

\subsection{Eigenvalues of $MM^\top$ in Subsubsections \ref{subsubsec:cubes} and \ref{subsubsec:crosspolytopes}}\label{app:cubescrosspolytopes_eigv}

Notice that for any matrices $A$, $AA^\top\succeq 0$ and $A^\top A\succeq 0$ have the same nonzero eigenvalues.  Thus, for Subsubsections~\ref{subsubsec:cubes} and \ref{subsubsec:crosspolytopes}, we only need to compute the largest eigenvalue of the matrix
\begin{equation*}
    N = \left[m_1\cdots m_d~-m_1\cdots -m_d\right]^\top \left[m_1\cdots m_d~-m_1\cdots -m_d\right] = 2^d \left[\widetilde{m}_1\cdots \widetilde{m}_d~-\widetilde{m}_1\cdots -\widetilde{m}_d\right],
\end{equation*}
where
\begin{equation*}
    m_i=\left[\underbrace{\underbrace{1\cdots 1}_{\text{length}=2^{d-i}}\underbrace{-1\cdots -1}_{\text{length}=2^{d-i}}}_{\text{repeat}~2^{i-1}~\text{times}}\cdots\right]^\top, i=1,\ldots,d
\end{equation*}
and
\begin{equation*}
    \left(\widetilde{m}_i\right)_j = \begin{cases}
        1,~\text{if $j=i$},\\
        -1,~\text{if $j=d+i$},\\
        0,~\text{if otherwise},\\
    \end{cases} 
    i=1,\ldots,d,~~j=1,\ldots,2d.
\end{equation*}

Suppose $x=[x_1\cdots x_{2d}]^\top$ with $\|x\|=1$ is an eigenvector of $N$ corresponding to eigenvalue $\lambda$.  Then, we have
\begin{equation*}
    Nx = \lambda x = 2^d\begin{bmatrix}
        x_1-x_{d+1}\\
        \vdots\\
        x_d-x_{2d}\\
        -x_1+x_{d+1}\\
        \vdots\\
        -x_d+x_{2d}\\
    \end{bmatrix}
\end{equation*}
which gives
\begin{equation*}
    \lambda = 2^d\sqrt{2\sum\limits_{i=1}^d\left(x_i-x_{d+i}\right)^2}.
\end{equation*}
If $x$ corresponds to the largest eigenvalue $\lambda_{\max}(N)$, then $x$ satisfy
\begin{equation*}
    x_{d+i} = -x_i,i=1,\ldots,d.
\end{equation*}
From $\|x\|=1$, we get
\begin{equation*}
    x = \left[\frac{1}{\sqrt{2d}}\cdots\frac{1}{\sqrt{2d}}~-\frac{1}{\sqrt{2d}}\cdots -\frac{1}{\sqrt{2d}}\right]^\top
\end{equation*}
and so
\begin{equation*}
    \lambda_{\max}(N) = 2^d\sqrt{2d\frac{2}{d}} = 2^{d+1}.
\end{equation*}

\section{More numerical results}\label{app:morenumres}
The following tables give the exact proportions of our inscribable test sets for which inscriptions were found.

\subsection{Results on random polytopes}~


    \begin{table}[!htb]
    \caption{Accuracy in higher dimensions (rank tolerance 1e-6)}
    \centering
    \begin{tabular}{cccccccccccc}
    \hline
    Alg$\backslash$Para & $n8d5$ & $n9d5$ & $n10d5$ & $n8d6$ & $n9d6$ & $n10d6$\\ \hline
    SDP-$\lambda^c$ & 52/100 & 54/100 & 36/100 & 49/100 & 43/100 & 26/100\\
    SAP-$\lambda^c$ & 52/100 & 88/100 & 89/100 & 49/100 & 52/100 & 87/100\\
    SDP-$\lambda^h$ & 70/100 & 84/100 & 78/100 & 78/100 & 81/100 & 93/100\\
    SAP-$\lambda^h$ & 70/100 & 84/100 & 78/100 & 78/100 & 81/100 & 93/100\\
    SDP-$\lambda^*$ & 92/100 & 66/100 & 63/100 & 96/100 & 79/100 & 61/100\\
    SAP-$\lambda^*$ & 99/100 & 81/100 & 76/100 & 98/100 & 90/100 & 77/100\\
    AP-$\lambda^*$ & 100/100 & 96/100 & 98/100 & 99/100 & 100/100 & 100/100\\
    \hline
    \end{tabular}
    \end{table}\FloatBarrier

    \begin{table}[!htb]
    \caption{Accuracy in higher dimensions (rank tolerance 1e-6) continued}
    \centering
    \begin{tabular}{cccccccccccc}
    \hline
    Alg$\backslash$Para & $n9d7$ & $n10d7$ & $n10d8$\\ \hline
    SDP-$\lambda^c$ & 32/100 & 30/100 & 34/100\\
    SAP-$\lambda^c$ & 56/100 & 35/100 & 34/100\\
    SDP-$\lambda^h$ & 90/100 & 88/100 & 81/100\\
    SAP-$\lambda^h$ & 90/100 & 88/100 & 81/100\\
    SDP-$\lambda^*$ & 96/100 & 83/100 & 98/100\\
    SAP-$\lambda^*$ & 99/100 & 88/100 & 99/100\\
    AP-$\lambda^*$ & 100/100 & 100/100 & 100/100\\
    \hline
    \end{tabular}
    \end{table}\FloatBarrier

\subsection{Results on the list of inscribable polytopes}~
\begin{table}[!htb]
    \caption{Accuracy (rank tolerance 1e-6)}
    \centering
    \begin{tabular}{ccccccc}
    \hline
    Alg$\backslash n$ & 5 & 6 & 7 & 8 & 9 & 10 \\ \hline
    SDP-$\lambda^h$ & 1/1 & 2/2 & 2/5 & 29/37 & 438/1140 & 31/100 \\
    SAP-$\lambda^h$ & 1/1 & 2/2 & 2/5 & 29/37 & 438/1140 & 31/100 \\
    SDP-$\lambda^*$ & 1/1 & 2/2 & 4/5 & 31/37 & 994/1140 & 61/100\\ 
    SAP-$\lambda^*$ & 1/1 & 2/2 & 4/5 & 33/37 & 1059/1140 & 71/100\\ 
    AP-$\lambda^*$ & 1/1 & 2/2 & 5/5 & 35/37 & 1115/1140 & 96/100\\ 
    \hline
    \end{tabular}
\end{table}\FloatBarrier
\end{document}